\documentclass[12pt]{amsart}

\usepackage{delimset, amsrefs, xcolor}
\usepackage[colorlinks, citecolor=cyan]{hyperref}
\usepackage{tikz, tkz-berge}
\usepackage{enumitem}
\setlist[enumerate,itemize]{itemsep=3pt}
\usepackage[capitalize]{cleveref}

\newtheorem{thm}{Theorem}[section]
\newtheorem{conj}[thm]{Conjecture}
\newtheorem{cor}[thm]{Corollary}

\newtheorem{prop}[thm]{Proposition}

\theoremstyle{remark}
\newtheorem{remark}{Remark}[section]

\DeclareMathOperator{\des}{\mathrm{des}}
\DeclareMathOperator{\asc}{\mathrm{asc}}
\DeclareMathOperator{\cdes}{\mathrm{cdes}}
\DeclareMathOperator{\casc}{\mathrm{casc}}
\DeclareMathOperator{\std}{\mathrm{std}}
\DeclareMathOperator{\pk}{\mathrm{pk}}

\DeclareMathOperator{\st}{\mathrm{st}}
\DeclareMathOperator{\depth}{\mathrm{dp}}
\DeclareMathOperator{\D}{\mathrm{D}}
\newcommand\floor[1]{\lfloor#1\rfloor}
\usepackage{mathrsfs} 
\newcommand{\B}{\mathscr{B}}
\renewcommand{\P}{\mathscr{P}}
\renewcommand{\O}{\mathscr{O}}

\newcommand{\Eulerian}[2]{\genfrac{<}{>}{0pt}{1}{#1}{#2}}

\crefname{thm}{Theorem}{Theorems}
\creflabelformat{thm}{~\upshape #2#1#3}

\numberwithin{equation}{section}
\numberwithin{figure}{section}
\numberwithin{table}{section}

\allowbreak
\allowdisplaybreaks

\title{The peak and descent statistics over ballot permutations}
\author[D.G.L. Wang]{David G.L. Wang$^{1,2}$}
\address{
$^1$School of Mathematics and Statistics, Beijing Institute of
Technology, 102488 Beijing, P. R. China\\
$^2$Beijing Key Laboratory on MCAACI, Beijing Institute of
Technology, 102488 Beijing, P. R. China
}
\email{glw@bit.edu.cn}

\author[Tongyuan Zhao]{Tongyuan Zhao$^3$}
\address{
$^3$College of Sciences,
China University of Petroleum,
102249 Beijing, P. R. China}
\email{zhaotongyuan@cup.edu.cn}

\thanks{Wang is supported by the General Program of National Natural
Science Foundation of China (Grant No.\ 11671037). Zhao is Supported
by Science Foundation of China University of Petroleum, Beijing
(Grant No.\ 2462020YXZZ004).}



\begin{document}
\maketitle

\begin{abstract}
A ballot permutation is a permutation $\pi$ such that in any prefix of $\pi$ the descent number is not more than the ascent number. By using a reversal-concatenation map, we give a formula for the joint distribution (pk, des) of the peak and descent statistics over ballot permutations, and connect this distribution and the joint distribution (pk, dp, des) of the peak, depth, and descent statistics over ordinary permutations in terms of generating functions. As corollaries, we obtain several formulas for the bivariate generating function for (i) the peak statistic over ballot permutations, (ii) the descent statistic over ballot permutations,
and (iii) the depth statistic over ordinary permutations. In particular, we confirm Spiro’s conjecture which finds the equidistribution of the descent statistic for ballot permutations and an analogue of the descent statistic for odd order permutations.
\bigskip

\end{abstract}


\section{Introduction}

In 1887 Bertrand~\cite{Ber87} introduced the \emph{ballot problem}: Consider an election for two candidates $A$ and $B$ with a total of $n$ votes, where $A$ wins $a$ votes and $B$ wins $n-a=b$ votes. What is the probability that at each count $A$ is always ahead? Equivalently,
what is the probability of a lattice path from the origin to the point $(b,\,a-1)$
is a ballot path?
Here a \emph{ballot path}
is a lattice path that never goes below the line $y=x$, see \cite{CFJ11,MSTT16,NS10,Whi78,Ber87}.
This is one of the beginnings of lattice path enumeration and early problems in probabilistic combinatorics, see Humphreys~\cite{Hum10}.
Among various ways of solving the ballot problem,
there are the well known reflection principle~\cite{Hum10}
and the cycle lemma~\cite{DM47}.
The answer to the ballot problem is the \emph{ballot number}
\[
\frac{a-b}{a+b}{a+b\choose b},
\]
which reduces to the Catalan number $\frac{1}{n+1}{2n\choose n}$
when $a=n+1$ and $b=n$, see~\cite{Aig08,Ges92}.
The ballot problem was generalized by Barbier~\cite{Bar87}
which demands $A$ maintains as more than $k$ times many votes as $B$,
see also Renault \cite{Ren07,Tak97}.

A \emph{ballot permutation} is a permutation
in any prefix of which the number of ascents
is at least the number of descents.
Every ballot permutation on $n$ distinct integers
naturally corresponds to a ballot path $p_{1}p_{2}\dotsm p_{n}$ such that $p_{i}p_{i+1}$ is the unit north step $(0,1)$ if $i$ is an ascent,
and the unit east step $(1,0)$ otherwise.

The problem of enumerating ballot permutations is closely related with
that of enumerating ordinary permutations with
a given up–down signature, see \cite{And81,BFW07,Niv68,She12}.
Bernardi, Duplantier and Nadeau \cite{BDN10} proved
that the number of ballot permutations of length $n$
equals the number of odd order permutations of length $n$,
by using compositions of bijections, and thus \cref{thm:BDN} follows.
A short proof is given by the first author and Zhang~\cite{WZ20}.

\begin{thm}[Bernardi, Duplantier and Nadeau]\label{thm:BDN}
The number of ballot permutations of length $n$ is
\[
b_n=\begin{cases}
[(n-1)!!]^{2},&\text{if $n$ is even},\\
n!!(n-2)!!,&\text{if $n$ is odd},
\end{cases}
\]
where $(-1)!!=1$.
\end{thm}
The sequence $\{b_n\}_{n\ge0}$ can be found in OEIS \cite[A000246]{OEIS}, of which
the exponential generating function is
\begin{equation}\label{gf:b}
\sum_{n\geq0}b_n \frac{x^n}{n!}
=\sqrt{\frac{1+x}{1-x}}.
\end{equation}
A ballot permutation whose corresponding ballot path
ends on the line $y=x$ is said to be a \emph{Dyck permutation},
whose enumeration is the Eulerian-Catalan number, see Bidkhori and Sullivant~\cite{BS11}.

Spiro~\cite{Spi20} introduced a statistic $M(\pi)$ for odd order permutations,
and conjectured that
the number of ballot permutations of length $n$ with $d$ descents
equals the number of odd order permutations $\pi$ of length $n$
such that $M(\pi)=d$.
In this paper,
we confirm the conjecture by computing their bivariate generating functions
in terms of the Eulerian numbers, respectively.

\begin{thm}\label{thm:Spiro}
Let $n\ge 1$ and $0\le d\le \floor{(n-1)/2}$.
The number of ballot permutations of length $n$
with $d$ descents equals the number of odd order permutations $\pi$
of length $n$ with $M(\pi)=d$.
\end{thm}

The first author and Zhang~\cite{WZ20}
refined Spiro’s conjecture by tracking the neighbors of the
largest letter in these permutations, which is still open.
They defined a word $u$ as a \emph{factor} of a word $w$ if there exist words $x$ and $y$ such that
$w = xuy$, and a word $u$ as a \emph{cyclic factor} of a permutation $\pi\in\mathcal{S}_n$ if $u$ is a factor of some word $v$ such that $(v)$ is a cycle of $\pi$. The conjecture is as follows.

\begin{conj}[Wang and Zhang]\label{thm:WZ}
For all $n$, $d$, and $2\leq j \leq n-1$, we have $b_{n,d}(1,j)+b_{n,d}(j,1)= 2p_{n,d}(1,j)$,
where $b_{n,d}(i,j)$ is the number of ballot permutations of length $n$ with $d$ descents which have $inj$ as a factor, and $p_{n,d}(i,j)$ is the number of odd order permutations of length $n$ with $M(\pi)=d$ which have $inj$ as a cyclic factor.
\end{conj}


Manes, Sapounakis, Tasoulas and Tsikouras~\cite{MSTT11} introduced
the concept of depth for a lattice path defined to be the difference
between the height of the lowest position of the path
and that of the starting point.
In this paper,
we also consider the depth statistic of a ballot permutation $\pi$ defined to be
the depth of the ballot path corresponding to~$\pi$.
Under this notion, a permutation is a ballot one if and only if its depth is zero.

Zhuang~\cite{Zhuang17} studied the generating function $P^{(\pk,\,\des)}$ of
the peak number and descent number over ordinary permutations
using noncommutative symmetric functions.
By using a map which we called
\emph{reversal-concatenation},
and an operator tool,
we find a relation between $P^{(\pk,\,\des)}$ and the generating function $B^{(\pk,\,\des)}$ of the same statistics over ballot permutations, see \cref{RelPkDes}.
The reversal-concatenation enables us to deal with the relations between the joint distribution (pk, dp, des) over ordinary permutations
and the joint distribution (pk, des) over ballot permutations in a uniform manner
which derive several corollaries, see \cref{sec:stat:ballot}.

The discovery of the map was inspired by Gessel's combinatorial interpretation of a decomposition of formal Laurent series in terms of lattice paths \cite{Ges80},
Bernardi et~al.'s path decompositions~\cite{BDN10},
and the $\omega$-decomposition~\cite{WZ20}.
Earlier decompositions were described by Feller~\cite[page~383]{Fel66}
and Foata and Sch\"utzenberger~\cite{FS71}.
The map is also related to the lowest points of the paths,
which was used in the studies of the Chung-Feller theorem,
see Woan~\cite{Woa01}, Shapiro~\cite{Sha01} and Eu, Fu and Yeh~\cite{EFY05}.
Basic generating functions calculating is also used throughout the paper, see Wilf~\cite{Wilf06B} and Flajolet and Sedgewick~\cite{FS09B}
for general techniques of generating functions.

The main results of this paper, besides~\cref{thm:Spiro},
are \cref{formdespk,recbdespkmh}.
In \cref{formdespk} we express the generating function $B^{(\pk,\,\des)}(x,y,t)$
as the image of a rational function under an operator.
In \cref{recbdespkmh} we provide a relation between the generating functions
$B^{(\pk,\,\des)}(x,y,t)$ and $P^{(\pk,\,\depth,\,\des)}(x,y,z,t)$.
As corollaries, we obtain the bivariate generating functions for peak number
over ballot permutations, and that for descent number
over ballot permutations, see~\cref{gf:bpk,fomofbxt}.

The next section consists of necessary notion and notation.
In \cref{sec:stat:ballot} we demonstrate the reversal-concatenation map
and its consequences.
In \cref{sec:proof:SpiroConj}, we use \cref{fomofbxt} to establish \cref{thm:Spiro}.

\section{Preliminary}\label{sec:preliminary}
Let $\mathcal S_{n}$ be the permutation group on the set $[n]=\{1, 2, \dots , n\}$.
A position $1\le i\le n-1$ in a permutation $\pi=\pi_1\pi_2\dotsm\pi_n\in\mathcal S_n$ is a \emph{descent} if $\pi_i>\pi_{i+1}$,
and an \emph{ascent} if $\pi_i<\pi_{i+1}$.
Denote the number of descents of $\pi$ by $\des(\pi)$,
and the number of ascents by $\asc(\pi)$.
We call the number
\[
h(\pi)
=\asc(\pi)-\des(\pi)
\]
the \emph{height} of $\pi$.
The permutation $\pi$ is said to be a \emph{ballot permutation}
if the height of any prefix of~$\pi$ is nonnegative, namely,
$h(\pi_1\pi_2\cdots\pi_i)\ge 0$
for all $i\in[n]$.
Let $\B_n$ denote the set of ballot permutations on $[n]$. Define $\B_{0}=\{\epsilon\}$,
where $\epsilon$ is the empty permutation.
The number of ballot permutations of height $0$, which are also called the \emph{Dyck permutations}, is the Eulerian-Catalan number.
We define
\[
\depth(\pi)
:= -\textrm{min}\{h(\pi_{1}\pi_{2}\dotsm\pi_{i})\colon 1\leq i\leq n\},
\]
and $\depth(\epsilon)=0$.
It is clear that
\[
0\leq \depth(\pi)\leq \des(\pi).
\]
A \emph{lowest position} of $\pi$ is a position $1\le i\le n$ such that $h(\pi_1\pi_2\dotsm \pi_i)=-\depth(\pi)$. Denote by $L(\pi)$ the set of
lowest positions of $\pi$. For example,
\[
\depth(5641327)=1
\quad\text{and}\quad
L(5641327)=\{4,6\}.
\]
From the definition, we see that
\[
\pi\in\B_n
\iff
1\in L(\pi).
\]

Let $\O_n$ be the set of \emph{odd order permutations}
of $[n]$, viz., the set of permutations of $[n]$ which are the products of cycles with odd lengths. In order to define an analogue for the descent statistic in the context
of odd order permutations, Spiro~\cite{Spi20} defines
for a permutation $\pi$ that
\[
M(\pi)=\sum_{c}\min\brk1{\cdes(c),\,\casc(c)},
\]
where the sum runs over all cycles $c=(c_1c_2\dotsm c_k)$ of $\pi$, with the \emph{cyclic descent}
\[
\cdes(c)=\abs{\{i\in[k]\colon c_i>c_{i+1}\ \text{where $c_{k+1}=c_1$}\}},
\]
and the \emph{cyclic ascent}
\[
\casc(c)=\abs{\{i\in[k]\colon c_i<c_{i+1}\ \text{where $c_{k+1}=c_1$}\}}
=\abs{c}-\cdes(c),
\]
where $\abs{c}$ is the length of $c$.

For $n\ge1$ and $0\le d\le n-1$,
the \emph{Eulerian number}, denoted as $E(n,d)$ or $\Eulerian{n}{d}$,
is the number of
permutations of $[n]$ with $d$ descents, namely
\[
E(n,d)=|\{\pi\in\mathcal S_n\colon \des(\pi)=d\}|,
\]
see OEIS~\cite[A008292]{OEIS}.
We adopt the convention $E(0,0)=1$ and
\[
E(n,d)=0,\qquad\text{if $n<0$, or $d<0$, or $d=n\ge 1$, or $d>n$}.
\]
As will be seen,
this extension helps dealing with summation calculation
by simplifying the domain of indices in summations,
so that one may focus on the summands.
For instance,
the notation $\sum_i$ implies that the index $i$ runs over all integers.

The $n$th \emph{Eulerian polynomial} is
\[
A_n(t)=\sum_{\pi\in\mathcal S_n}t^{\des(\pi)}
=\sum_{d}E(n,d)t^{d}\quad\text{for $n\ge 1$},
\]
and $A_0(t)=1$,
see Kyle Petersen~\cite[\S 1.4]{Kyle15B}.
The exponential generating function of the Eulerian polynomials is
\begin{equation}\label{gf:A}
\sum_{n}A_{n}(t)\frac{x^{n}}{n!}
=\frac{t-1}{t-e^{(t-1)x}},
\end{equation}
see \cite[Theorem 1.6]{Kyle15B} and
\cite[Formula (75)]{FS09B}. (It is often useful to consider the variant $E_n(t)$ of Eulerian polynomials defined by $E_n(t)=tA_n(t)$, which are also called the \emph{Eulerian polynomials} in some literatures, see Stanley~\cite[\S 1.3]{Sta97B} and B\'ona~\cite[Theorem 1.22]{Bona12B}.)

The bivariate generating function
\begin{align*}
E(x,t)
&=\sum_{n\ge1}\sum_{d}\frac{E(n,d)t^{d}x^{n}}{n!}
=\sum_{n\ge1}A_n(t)\frac{x^{n}}{n!}\\
&=x+\frac{x^2}{2}(1+t)+\frac{x^3}{3!}\brk1{1+4t+t^2}
+\frac{x^4}{4!}\brk1{1+11t+11t^2+t^3}+\dotsm
\end{align*}
has the closed form
\begin{equation}\label{Eulequ}
E(x,t)
=\frac{t-1}{t-e^{(t-1)x}}-1
=\frac{e^{(1-t)x}-1}{1-t e^{(1-t)x}}.
\end{equation}

For $m\geq 1$, nonnegative integers $n_1,n_2,\dots,n_m$, and
statistics
$\st_{1},\,\st_{2},\,\dots,\,\st_{m}$ over~$\mathcal S_{n}$,
let
\begin{align*}
\P_n^{(\st_{1},\,\st_{2},\,\dots,\,\st_{m})}(n_1,n_2,\dots,n_m)
&=\{\pi\in\mathcal S_{n}\colon\st_{i}(\pi)= n_{i}\ \text{for all $1\leq i\leq m$}\}
\quad\text{and}\quad\\
p_n^{(\st_{1},\,\st_{2},\,\dots,\,\st_{m})}(n_1,n_2,\dots,n_m)
&=\abs{\P_n^{(\st_{1},\,\st_{2},\,\dots,\,\st_{m})}(n_1,n_2,\dots,n_m)}.
\end{align*}
For convenience, we define
$\P_n^{(\st_{1},\,\st_{2},\,\dots,\,\st_{m})}(n_1,n_2,\dots,n_m)=\emptyset$
if any one of $n,n_1,\dots,n_m$ is negative, and
\[
\P_0^{(\st_{1},\,\st_{2},\,\dots,\,\st_{m})}(n_1,n_2,\dots,n_m)=\begin{cases}
\{\epsilon\},&\text{if $n_1=n_2=\dotsb=n_m=0$,}\\
\emptyset,&\text{otherwise}.
\end{cases}
\]
We study the joint distribution of the statistics $\st_{1}$, $\st_{2}$, $\dots$, $\st_{m}$ by the generating function
\begin{align*}
&P^{(\st_{1},\,\st_{2},\,\dots,\,\st_{m})}(x,x_{1},\dots,x_{m})\\
=\,&\sum_{n,n_1,\dots,n_m}p_n^{(\st_{1},\,\st_{2},\,\dots,\,\st_{m})}
(n_1,n_2,\dots,n_m)x_{1}^{n_{1}}x_{2}^{n_{2}}\dotsm x_{m}^{n_{m}}\frac{x^{n}}{n!}\\
=\,&1+\sum_{n\geq 1}p^{(\st_{1},\,\st_{2},\,\dots,\,\st_{m})}_{n}
(x_{1},x_{2},\dots,x_{m})\frac{x^{n}}{n!}.
\end{align*}
Replacing the set $\mathcal S_n$ by $\B_n$ in the above definitions,
we obtain analogous definitions of the set
$\B_n^{(\st_{1},\,\dots,\,\st_{m})}(n_1,\dots,n_m)$,
the number $b_n^{(\st_{1},\,\dots,\,\st_{m})}(n_1,\dots,n_m)$,
and the generating function
$B^{(\st_{1},\,\dots,\,\st_{m})}(x,x_{1},\dots,x_{m})$.
For convenience,
we denote $(\st_1)$ by $\st_1$ without the parentheses.

For a permutation $\pi=\pi_1\pi_2\dotsm\pi_n$ on distinct integers,
denote by $\pi^r$ the reversal of $\pi$, namely,
$\pi^r=\pi_n\pi_{n-1}\dotsm\pi_1$.
The \emph{standardization} of $\pi$, denoted $\std(\pi)$,
is the permutation $\sigma_{1}\sigma_2\dotsm\sigma_{n}\in \mathcal S_{n}$
such that $\sigma_{i}<\sigma_{j}$ if and only if $\pi_{i}<\pi_{j}$.
For convenience, we define $\std(\epsilon)=\epsilon$.

\section{Statistics over ballot permutations}\label{sec:stat:ballot}

In this section, we use a map, which we call the \emph{reversal-concatenation},
to establish a series of relations between joint distributions of statistics
over $\mathcal S_n$ and $\B_n$.
Let
\[
\mathcal P=\{(\rho,\tau)\colon  \text{$\exists\ 0\leq l \leq n$ such that
$\std(\rho)\in \B_{l}$,
$\std(\tau)\in \B_{n-l}$,
and $\rho\tau\in \mathcal S_{n}$}\}.
\]
The reversal-concatenation map $\phi$ is defined as
\begin{align}
\phi\colon\mathcal P&\to\bigcup_{n\geq1}\mathcal S_{n}\notag\\
(\rho,\tau)&\mapsto\rho^{r}\tau.\label{def:varphi}
\end{align}
Suppose that
$\rho=\rho_1\rho_2\cdots\rho_{l}$ and
$\tau=\tau_1\tau_2\cdots\tau_{n-l}$.
Let $\pi=\phi(\rho,\tau)$.
It is easy to check that
\begin{equation}\label{des:pi:rho1<tau1}
\des(\pi)
=l-1-\des(\rho)+\des(\tau)+\chi(\text{$\rho=\epsilon$ or $\rho_{1}>\tau_1$}),
\end{equation}
where $\chi$ is the characteristic function.
It is also easy to see that when $(\rho,\tau)\neq(\epsilon,\epsilon)$,
\begin{itemize}
\item
if $\rho=\epsilon$ or $\rho_1>\tau_1$,
then the position $l+1$ is the first lowest position of $\pi$;
\item
if $\tau=\epsilon$ or $\rho_1<\tau_1$,
then the position $l$ is the last lowest position of $\pi$.
\end{itemize}
For example,
the first lowest position of the permutation $\phi(341,265)=143265$ is 4;
the last lowest position of $\phi(134,256)=431256$ is 3.

A position $2 \leq i \leq n-1$ in a permutation $\pi=\pi_{1}\pi_{2}\dotsm \pi_{n}\in\mathcal S_n$ is a \emph{peak} if $\pi_{i-1}<\pi_{i}$ and $\pi_i>\pi_{i+1}$.
Let $\pk(\pi)$ be the number of peaks in $\pi$.

\subsection{Connecting the joint distribution (pk, dp, des) over $\mathcal S_n$
and the joint distribution (pk, des) over $\B_n$}
The goal of this subsection is to establish \cref{formdespk,recbdespkmh}.
For clarification, we show the following specification of \cref{recbdespkmh} first.

\begin{thm} \label{recbdespk}
For $n \geq 0,k\geq 0$ and $d\geq 0$ such that $(n,k,d) \neq (0,0,1)$,
we have
\begin{equation}\label{e17}
p_n^{(\pk,\,\des)}(k,d)+p_n^{(\pk,\,\des)}(k,\,d-1)
=\sum_{l,\,i,\,j}\binom{n}{l}\,b_l^{(\pk,\,\des)}(i,j)
b_{n-l}^{(\pk,\,\des)}(k-i,\,d-l+j).
\end{equation}
\end{thm}

\begin{proof}
It is direct to check \cref{e17} for $n\leq d$ and for $d=0$.
Let $n>d\ge 1$ and
\begin{multline*}
R_n(k,d)
=\{(\rho,\tau)\colon\text{$\exists~0\leq l \leq n$ and $i,j\geq0$ such that}\\
\std(\rho)\in \B_{l}^{(\pk,\,\des)}(i,j),\
\std(\tau)\in \B_{n-l}^{(\pk,\,\des)}(k-i,\,d-l+j),\
\text{and }\rho\tau\in \mathcal S_{n}\}.
\end{multline*}
We shall show that both sides of \cref{e17} equal $\abs{R_n(k,d)}$.

For the left side, we shall show that
$\varphi=\phi|_{R_n(k,d)}\colon R_n(k,d)\to U_{n}(k,d)$
is a bijection, where
\[
U_{n}(k,d)=\P_n^{(\pk,\,\des)}(k,d)\cup \P_n^{(\pk,\,\des)}(k,\,d-1).
\]
First, we verify that $\varphi(R_n(k,d))\subseteq U_{n}(k,d)$.
Let $(\rho,\tau)\in R_n(k,d)$. Suppose that
\[
\rho=\rho_1\rho_2\dotsm\rho_l
\quad\text{and}\quad
\tau=\tau_{l+1}\tau_{l+2}\dotsm\tau_{n},
\]
\[
\std(\rho)\in \B_{l}^{(\pk,\,\des)}(i,j)
\quad\text{and}\quad
\std(\tau)\in \B_{n-l}^{(\pk,\,\des)}(k-i,\,d-l+j).
\]
Let $\pi=\varphi(\rho,\tau)$, i.e., $\pi=\rho^{r}\tau$.
\begin{itemize}
\item
If $\rho=\epsilon$ or $\rho_1>\tau_1$, then the first
lowest position of $\pi$ is $l+1$. By \cref{des:pi:rho1<tau1},
\[
\des(\pi)
=(l-1-j)+(d-l+j)+1=d.
\]
Since every peak in $\pi$ is either in $\rho^{r}$ or in $\tau$, we find
\[
\pk(\pi)
=\pk(\rho^{r})+\pk(\tau)
=\pk(\rho)+\pk(\tau)
=i+(k-i)=k.
\]
Thus $\pi\in\P_n^{(\pk,\,\des)}(k,d)$.
\item
If $\tau=\epsilon$ or $\rho_1<\tau_1$,
then the last lowest position of $\pi$ is $l$. Similarly,
\[
\des(\pi)=(l-1-j)+(d-l+j)=d-1
\quad\text{and}\quad
\pk(\rho^{r}\tau)=\pk(\rho)+\pk(\tau)=k.
\]
Thus $\pi\in\P_n^{(\pk,\,\des)}(k,\,d-1)$.
This completes the verification.
\end{itemize}

Second, we show that $\varphi$ is injective.
Suppose that $(\rho',\tau')\in R_n(k,d)$ such that
$\varphi(\rho',\tau')=\pi$, i.e., $(\rho')^{r}\tau'=\rho^{r}\tau$.
For $\pi\in U_{n}(k,d)$, we define
\[
l(\pi)=\begin{cases}
\max L(\pi),&\text{if $\des(\pi)=d-1$},\\
\min L(\pi)-1,&\text{if $\des(\pi)=d$}.
\end{cases}
\]
By the definition of $\varphi$,
both the permutations $\rho$ and $\rho'$ have length $l(\pi)$.
It follows immediately that $\rho'=\rho$ and $\tau'=\tau$.
This proves the injectiveness of $\varphi$.

Thirdly, we show that $\varphi$ is surjective.
For any $\pi=\pi_1\pi_2\dotsm \pi_n\in U_{n}(k,d)$,
let
\[
\rho=\pi_{l(\pi)}\pi_{l(\pi)-1}\cdots\pi_{1}
\quad\text{and}\quad
\tau=\pi_{l(\pi)+1}\pi_{l(\pi)+2}\cdots\pi_{n}.
\]
It is clear that $\varphi(\rho,\tau)=\pi$.
Hence $\varphi$ is surjective and thus bijective.
Therefore,
\[
\abs{R_n(k,d)}=\abs{U_{n}(k,d)}
=p_n^{(\pk,\,\des)}(k,d)+p_n^{(\pk,\,\des)}(k,\,d-1).
\]

Now we show the right side of \cref{e17} also equals $\abs{R_n(k,d)}$.
In fact,
since for any pair $(\rho',\tau')\in\B_l^{(\pk,\,\des)}(i,j)\times\B_{n-l}^{(\pk,\,\des)}(k-i,\,d-l+j)$,
there are ${n\choose l}$ pairs $(\rho,\tau)$ such that
$\std(\rho)=\rho'$ and $\std(\tau)=\tau'$, we obtain
\[
\abs{R_n(k,d)}
=\bigcup_{l,\,i,\,j}
\binom{n}{l}\,\abs1{\B_l^{(\pk,\,\des)}(i,j)\times \B_{n-l}^{(\pk,\,\des)}(k-i,\,d-l+j)},
\]
which is simplified to the right side of \cref{e17}.
This completes the proof.
\end{proof}

\begin{remark} \label{rem1}
From the proof of \cref{recbdespk},
we see that similar statements hold
if the statistic pk is replaced by a statistic st such that
\begin{enumerate}
\item
$\textrm{st}(\epsilon)=0$;
\item
$\st(\pi)=\st(\pi^{r})$ for any $\pi$;
\item
$\st(\pi)=\st(\pi_{1}\pi_{2}\dotsm\pi_{i-1})+\st(\pi_{i}\pi_{i+1}\dotsm\pi_{n})$
for any $\pi=\pi_{1}\pi_{2}\dotsm\pi_{n}\in\mathcal S_{n}$,
where $i=\min L(\pi)$.
\end{enumerate}
\end{remark}

\cref{recbdespk} is translated into the language of generating functions as follows.
\begin{thm} \label{RelPkDes}
\[
B^{(\pk,\,\des)}\brk2{xt,\,y,\,\frac{1}{t}}B^{(\pk,\,\des)}(x,y,t)
=(1+t)P^{(\pk,\,\des)}(x,y,t)-t.
\]
\end{thm}
\begin{proof}
Multiplying each term in \cref{e17} by $y^{k}t^{d}x^{n}/n!$ and
summing over all integers $n \geq 0$, $d \geq 0$ and $k \geq 0$
such that $(n, k, d) \neq (0, 0, 1)$, we deduce the following equations respectively:
\begin{align*}
\sum_{(n, k, d) \neq (0, 0, 1)}p_n^{(\pk,\,\des)}(k, d) y^{k}t^{d}\,\frac{x^{n}}{n!}
&=\P^{(\pk,\,\des)}(x,y,t),\\
\sum_{(n, k, d) \neq (0, 0, 1)}p_n^{(\pk,\,\des)}(k, d-1) y^{k}t^{d}\,\frac{x^{n}}{n!}
&=t\sum_{(n, k, d) \neq (0, 0, 1)}p_n^{(\pk,\,\des)}(k, d-1)y^{k}t^{d-1}\,\frac{x^{n}}{n!} \\
&=t\sum_{(n, k, d) \neq (0, 0, 0)}p_n^{(\pk,\,\des)}(k, d) y^{k}t^{d}\,\frac{x^{n}}{n!}\\
&=t\,\P^{(\pk,\,\des)}(x,y,t)-t,
\end{align*}
and
\begin{align*}
&\sum_{(n, k, d) \neq (0, 0, 1)}\sum_{l,\,i,\,j}
\binom{n}{l}b_l^{(\pk,\,\des)}(i,j)
b_{n-l}^{(\pk,\,\des)}(k-i,\, d-l+j)y^{k}t^{d}\,\frac{x^{n}}{n!}\\
=&\sum_{n,\,k,\,d,\,l,\,i,\,j}
b_l^{(\pk,\,\des)}(i,j)\frac{y^{i}}{t^{j}}\frac{(xt)^{l}}{l!}
\cdot  b_{n-l}^{(\pk,\,\des)}(k-i,\, d-l+j)y^{k-i}t^{d-l+j}\frac{x^{n-l}}{(n-l)!} \\
=&\sum_{l,\,i,\,j}b_l^{(\pk,\,\des)}(i,\, j)\frac{y^{i}}{t^{j}}\frac{x^{l}}{l!}
\sum_{n,\,k,\,d}b_n^{(\pk,\,\des)}(k,\, d)y^{k} t^d\frac{x^{n}}{n!}\\
=&B^{(\pk,\,\des)}\brk2{xt,\,y,\,\frac{1}{t}}B^{(\pk,\,\des)}(x,y,t).
\end{align*}
Combining them together, we obtain the desired equation.
\end{proof}

In order to give a formula for the generating function $B^{(\pk,\,\des)}(x,y,t)$, we introduce an operator
$D^{x_1,x_2}$ for multivariate formal power series
\[
P=P(x_1,x_2,\dots,x_k)
=\sum_{n_1,\,n_2,\,\dots,\,n_k}p(n_1,n_2,\dots,n_k)
x_{1}^{n_1}x_{2}^{n_2}\dotsm x_{k}^{n_k}
\]
by defining
\[
D^{x_1,x_2}P
=\sum_{n_1\leq(n_2-1)/2,\,n_3,\,\dots,\,n_k}
p(n_1,n_2,\dots,n_k)x_{1}^{n_1}x_{2}^{n_2}\dotsm x_{k}^{n_k}.
\]
For example, $D^{t,x}(x+3x^{2}yt+2x^{3}y^{2}t)=x+2x^{3}y^{2}t$.
From the definition, it is easy to see that
\[
(D^{x_1,x_2}P)|_{x_{i_1}=c_1,\,\dots,\,x_{i_j}=c_j}
=D^{x_1,x_2}(P|_{x_{i_1}=c_1,\,\dots,\,x_{i_j}=c_j})
\]
for any $3\leq i_1<\dots<i_j\leq k$ and constants $c_1,c_2,\dots,c_j$.

Besides, we will need the following result of Zhuang~\cite[Theorem 4.2]{Zhuang17}.
\begin{thm} \label{zhuang}
For $n\ge 1$,
\begin{equation}\label{eq:despk}
\sum_{\pi\in\mathcal S_{n}}t^{\des(\pi)+1}y^{\pk(\pi)+1}
=\brk4{\frac{1+u}{1+uv}}^{n+1}vA_{n}(v),
\end{equation}
where
\begin{equation}\label{def:uv}
\begin{cases}
\displaystyle
u=\frac{1+t^{2}-2yt-(1-t)\sqrt{(1+t)^{2}-4yt}}{2(1-y)t},\\[12pt]
\displaystyle
v=\frac{(1+t)^{2}-2yt-(1+t)\sqrt{(1+t)^{2}-4yt}}{2yt}.
\end{cases}
\end{equation}
\end{thm}

Now we can give a formula for $B^{(\pk,\,\des)}(x,y,t)$.
\begin{thm} \label{formdespk}
\[
B^{(\pk,\,\des)}(x,y,t)
=\exp\brk4{\D^{t,x}\ln\brk3{1+\frac{(1+t)(1+u)v(w-1)}{yt(1+uv)(1-vw)}}}.
\]
where $u$ and $v$ are defined by \cref{def:uv}, and
\[
w=\exp\brk3{\frac{x(1+u)(1-v)}{1+uv}}.
\]
\end{thm}
\begin{proof}
By \cref{eq:despk} and \cref{gf:A}, we can deduce that
\begin{align*}\label{e18}
  P^{(\pk,\,\des)}(x,y,t)
  & =1+\sum_{n \geq 1}\sum_{\pi\in\mathcal S_{n}}y^{\pk(\pi)} t^{\des(\pi)}\frac{x^n}{n!} \\
  & =1+ \frac{1}{yt}\sum_{n\geq 1}\frac{(1+u)v}{1+uv}A_{n}(v)\frac{\brk2{\frac{x(1+u)}{1+uv}}^{n}}{n!} \\
  & =1+ \frac{(1+u)v}{yt(1+uv)}\brk4{\frac{v-1}{v-\exp\brk1{\frac{x(1+u)(v-1)}{1+uv}}}-1} \\
  &=1+\frac{(1+u)v(w-1)}{yt(1+uv)(1-vw)}.
\end{align*}
By \cref{RelPkDes}, we have
\begin{align*}
\ln\brk3{B^{(\pk,\,\des)}\brk2{xt,\,\frac{1}{t},\,y}}+
\ln\brk1{B^{(\pk,\,\des)}(x,y,t)}
&=\ln\brk1{(1+t)P^{(\pk,\,\des)}(x,y,t)-t} \\
&=\ln\brk3{1+\frac{(1+t)(1+u)v(w-1)}{yt(1+uv)(1-vw)}}.
\end{align*}
Since $\des(\pi)\leq (n-1)/2$ for $\pi\in\mathcal S_n$,
the expansion of
\[
\ln\brk1{B^{(\pk,\,\des)}(x,y,t)}
=\ln\brk3{1+\sum_{n\geq1}\sum_{\pi\in\mathcal S_n}y^{\pk(\pi)}t^{\des(\pi)}\frac{x^n}{n!}}
\]
is a multivariate formal power series,
with terms of the form $x^{n}y^{k}t^{d}$ such that $d \leq (n-1)/2$.
Similarly,
the terms of the series
\[
\ln\brk3{B^{(\pk,\,\des)}\brk2{xt,\,\frac{1}{t},\,y}}
=\ln\brk3{1+
\sum_{n\geq1}\sum_{\pi\in\mathcal S_n}y^{\pk(\pi)}t^{n-\des(\pi)}\frac{x^n}{n!}}
\]
are of the form $x^{n}y^{k}t^{d}$ such that $d>(n-1)/2$.
Therefore,
\[
\ln\brk1{B^{(\pk,\,\des)}(x,y,t)}
=\brk4{\D^{t,x}\ln\brk3{1+\frac{(1+t)(1+u)v(w-1)}{yt(1+uv)(1-vw)}}},
\]
which yields the desired equation.
\end{proof}

Now we give a generalization of \cref{RelPkDes},
by considering the statistic dp over $\mathcal S_n$.
It can be shown by a proof that is similar to those of \cref{recbdespk} and \cref{RelPkDes}.

\begin{thm} \label{recbdespkmh}
For $n,d,h,k\geq 0$ such that $(n,k,h,d) \neq (0,0,1,1)$,
\begin{multline} \label{e21}
p_n^{(\pk,\,\depth,\,\des)}(k,h,d)+p_n^{(\pk,\,\depth,\,\des)}(k,\,h-1,\,d-1) \\
=\sum_{i,\,j}\binom{n}{2i+h}
b_{2i+h}^{(\pk,\,\des)}(j,\,i)
b_{n-2i-h}^{(\pk,\,\des)}(k-j,\,d-i-h).
\end{multline}
In other words,
\begin{equation}\label{RelPkDesmh}
B^{(\pk,\,\des)}\brk2{xzt,\,y,\,\frac{1}{z^{2}t}}B^{(\pk,\,\des)}(x,y,t)
=(1+zt)P^{(\pk,\,\depth,\,\des)}(x,y,z,t)-zt.
\end{equation}
\end{thm}
\begin{proof}
Consider the set
\begin{multline*}
Q(n,k,d,h)=\{(\rho,\tau)\colon \rho\tau\in \mathcal S_{n},\ \exists\ i,j\geq0 \text{ such that }\\
\std(\rho)\in \B_{2i+h}^{(\pk,\,\des)}(j,i)\text{ and }\std(\tau)\in \B_{n-2i-h}^{(\pk,\,\des)}(k-j,\,d-i-h)\}.
\end{multline*}
Similar to the proof of \cref{recbdespk}, it can be proved that the map $\phi|_{Q(n,k,d,h)}$
is a bijection from $Q(n,k,d,h)$ to the union
\[
\P_n^{(\pk,\,\depth,\,\des)}(k,h,d)\cup \P_n^{(\pk,\,\depth,\,\des)}(k,\,h-1,\,d-1),
\]
which implies \cref{e21}.
The desired generating function can be obtained by using standard techniques
in generatingfunctionology as that is used in the proof of \cref{RelPkDes}.
\end{proof}

\begin{remark}
\cref{RelPkDesmh} reduces to \cref{RelPkDes} by specifying $z=1$.
\end{remark}

\subsection{The bivariate generating functions for statistics pk, des
over $\B_n$}\label{sec:gf:pk:des}
First, we can deduce the
bivariate generating functions for the statistic pk over $\B_n$.

\begin{thm}\label{gf:bpk}
\[
B^{\pk}(x,y)
=\sqrt{\frac{\sqrt{1-y}\cosh(x\sqrt{1-y})+\sinh(x\sqrt{1-y})}
{\sqrt{1-y}\cosh(x\sqrt{1-y})-\sinh(x\sqrt{1-y})}}.
\]
\end{thm}
\begin{proof}
We set $t=1$ in \cref{RelPkDes}.
Since
\[
B^{(\pk,\,\des)}(x,y,1)=B^{\pk}(x,y)
\quad\text{and}\quad
P^{(\pk,\,\des)}(x,y,1)=P^{\pk}(x,y),
\]
we find
\begin{equation}\label{e15}
[B^{\pk}(x,y)]^2=2P^{\pk}(x,y)-1.
\end{equation}
It is known that
\[
P^{\pk}(x,y)=\frac{\sqrt{1-y}\cosh(x\sqrt{1-y})}
{\sqrt{1-y}\cosh(x\sqrt{1-y})-\sinh(x\sqrt{1-y})},
\]
see Entringer~\cite{Ent69}, Kitaev~\cite{Kit07} and Zhuang~\cite{Zhuang16}
for instance.
Substituting the above equation into \cref{e15}, we derive the desired equation.
\end{proof}

The first few terms of $B^{\pk}(x,y)$ are as follows.
\begin{align*}
B^{\pk}(x,y)&=1+x+\frac{x^2}{2}+\frac{x^3}{3!}(1+2t)
+\frac{x^4}{4!}(1+8t)+\frac{x^5}{5!}\brk1{1+28t+16t^2}+\dotsb.
\end{align*}
The coefficients triangle of the above polynomial is not found in the OEIS~\cite{OEIS}.

Second, we can deduce
the bivariate generating function for the statistic des over $\B_n$.

\begin{thm}\label{fomofbxt}
\[
B^{\des}(x,t)
=\exp\brk3{x+2 \sum_{k\ge 1}\sum_{d\le k-1}
E(2k,d)t^{d+1}\frac{x^{2k+1}}{(2k+1)!}}.
\]
\end{thm}
\begin{proof}
Taking $y=1$ in \cref{RelPkDes}, since
\begin{align*}
B^{(\pk,\,\des)}\brk2{xt,\,1,\,\frac{1}{t}}
&=B^{\des}\brk2{xt,\,\frac{1}{t}},\\
B^{(\pk,\,\des)}(x,1,t)
&=B^{\des}(x,t),\quad\text{and}\quad\\
P^{(\pk,\,\des)}(x,1,t)
&=P^{\des}(x,t)=E(x,t)+1,
\end{align*}
we have
\[
B^{\des}\brk2{xt,\,\frac{1}{t}}B^{\des}(x,t)
=(1+t)\brk1{E(x,t)+1}-t
=1+(1+t)E(x,t).
\]
Similar to \cref{formdespk}, we have
\[
B^{\des}(x,t)=\exp\brk2{D^{t,x}\ln\brk1{1+(1+t)E(x,t)}}.
\]
By \cref{Eulequ},
\[
2\sum_{k\ge 1}\sum_{d}\frac{E(2k,d)t^{d}x^{2k}}{(2k)!}
=E(x,t)+E(t,\,-x)
=\frac{e^{(1-t)x}-1}{1-t e^{(1-t)x}}
+\frac{e^{(1-t)(-x)}-1}{1-t e^{(1-t)(-x)}}.
\]
Therefore,
\begin{align*}
2\sum_{k\ge 1}\sum_{d}\frac{E(2k,d)t^{d+1} x^{2k+1}}{(2k+1)!}
&=2t\int_{0}^{x} \sum_{k\ge 1}\sum_{d}
\frac{E(2k,d)t^{d}u^{2k}}{(2k)!}\, du\\
&=t\int_{0}^{x}
\brk3{\frac{e^{(1-t)u}-1}{1-t e^{(1-t)u}}
+\frac{e^{(1-t)(-u)}-1}{1-t e^{(1-t)(-u)}}}\,du\\
\label{e3}
& = \ln \frac{1-te^{x(t-1)}}{1-te^{(1-t)x}}-2xt.
\end{align*}
It is not difficult to check that
\[
\ln(1+(1+t)E(x,t))
=x-xt+\ln \frac{1-te^{x(t-1)}}{1-te^{(1-t)x}}.
\]
Thus
\begin{align*}
D^{t,x}\brk1{\ln(1+(1+t)E(x,t))}
&=D^{t,x}\brk3{x+xt+2\sum_{k\ge 1}\sum_{d}\frac{E(2k,d)t^{d+1} x^{2k+1}}{(2k+1)!}}\\
&=x+2 \sum_{k\ge 1}\sum_{d\le k-1}E(2k,d)t^{d+1}\frac{x^{2k+1}}{(2k+1)!},
\end{align*}
which completes the proof.
\end{proof}

Expanding the power series $B(x,t)$ in $x$, we obtain
\begin{multline*}
B^{\des}(x,t)
=1+x+\frac{x^2}{2}+\frac{x^3}{3!}(1+2t)
+\frac{x^4}{4!}(1+8t)
+\frac{x^5}{5!}\brk1{1+22t+22t^2}\\
+\frac{x^6}{6!}\brk1{1+52t+172t^2}
+\frac{x^7}{7!}\brk1{1+114t+856t^2+604t^3}
+\dotsb.
\end{multline*}
The coefficients triangle of the above polynomial is~\cite[A321280]{OEIS}.
Now, we can establish the bivariate generating function for the statistic dp over $\mathcal S_n$.

\begin{cor}\label{gf:dep}
\[
P^{\depth}(x,z)
=\frac{z}{1+z}+\frac{\sqrt{1-x^2}}{(1-x)(1+z)}
\exp\brk3{xz+2\sum_{k\ge 1}\sum_{d\le k-1}E(2k,\,k-1-d)z^{2d+1}\frac{x^{2k+1}}{(2k+1)!}}.
\]
\end{cor}
\begin{proof}
Taking $y=t=1$ in \cref{RelPkDesmh}, we obtain
\begin{equation}\label{pf:dep}
B^{\des}\brk2{xz,\frac{1}{z^2}}\brk3{\sum_{n\geq 0}b_n \frac{x^n}{n!}}
=(1+z)P^{\depth}(x,z)-z,
\end{equation}
where $b_n$ is the number of ballot permutations of length $n$.
Substituting $x$ by $xz$ in \cref{fomofbxt},
and then replacing $t$ by $1/z^2$, we obtain
\begin{align}
B^{\des}\brk2{xz,\,\frac{1}{z^2}}
&=\exp\brk3{xz+2 \sum_{k\ge 1}\sum_{d\le k-1} E(2k,d)z^{2k-2d-1}\frac{x^{2k+1}}{(2k+1)!}}\notag\\
&=\exp\brk3{xz+2\sum_{k\ge 1}\sum_{d\le k-1}E(2k,k-1-d)z^{2d+1}\frac{x^{2k+1}}{(2k+1)!}}. \label{pf:B:xz:1/z}
\end{align}
Substituting \cref{pf:B:xz:1/z,gf:b} into \cref{pf:dep},
one may solve $P^{\depth}(x,z)$ out as desired.
\end{proof}

The first few terms of $P^{\depth}(x,z)$ are as follows.
\begin{align*}
P^{\depth}(x,z)&=1+x+\frac{x^2}{2}+\frac{x^3}{3!}\brk1{3+2z+z^2}
+\frac{x^4}{4!}\brk1{9+11z+3z^2+z^3}+\dotsm.
\end{align*}
The coefficient triangle of the above polynomial is not found in the OEIS~\cite{OEIS}.

\section{A proof for \cref{thm:Spiro}}\label{sec:proof:SpiroConj}
For convenience, define $\O_0(0)=\{\epsilon\}$.
For any integer pair $(n,d)\ne(0,0)$, define
\begin{align*}
\O_n(d)=\{\pi \in \O_n\colon M(\pi)=d\}.
\end{align*}
Spiro's \cite[Proposition 3.2]{Spi20} can be restated as \cref{recofpnd}.

\begin{prop}[Spiro]\label{recofpnd}
For any integer $n\ge 0$ and any integer $d$,
\begin{equation}\label{e1}
\abs{\O_{n+1}(d)}
= \abs{\O_{n}(d)} + \sum_{i}\sum_{k\ge i}
2 {n \choose 2k} E(2k,\,i-1)\,|\O_{n-2k}(d-i)|.
\end{equation}
\end{prop}

Now we are in a position to prove \cref{thm:Spiro}.

\begin{proof}
In view of \cref{fomofbxt}, it is equivalent to show that
the generating function $O(x,t)=\sum_{n,d}\abs{\O_n(d)}t^{d}x^{n}/n!$ is
\begin{equation}\label{dsr:O}
O(x,t)=\exp\brk3{x+2 \sum_{k\ge 1}\sum_{d\le k-1}
E(2k,d)t^{d+1}\frac{x^{2k+1}}{(2k+1)!}}.
\end{equation}
In fact, multiplying each term in \cref{e1} by $t^{d}x^{n}/n!$ and
summing over all integers $n \geq 1$ and all integers $d$,
we deduce the following respectively:
\begin{align*}
\sum_{n\ge 1}\sum_{d}\abs{\O_{n+1}(d)} t^d\,\frac{x^n}{n!}
&=\frac{\partial O(x,t)}{\partial x}-\abs{\O_1(0)}-\abs{\O_1(1)}t
=\frac{\partial O(x,t)}{\partial x}-1,\\
\sum_{n\ge 1}\sum_{d}\abs{\O_n(d)}t^{d}\,\frac{x^{n}}{n!}
&=O(x,t)-1,
\end{align*}
and
\begin{align*}
& \sum_{n\ge 1}\sum_{d,\,i}\sum_{k\ge i}
2E(2k,\,i-1)\abs{\O_{n-2k}(d-i)}t^{d}\,\frac{x^{n}}{(2k)!(n-2k)!}\\
=\ &2t\,\sum_{i}\sum_{k\ge i}E(2k,\,i-1)t^{i-1}\frac{x^{2k}}{(2k)!}
\sum_{n,\,d}\abs{\O_{n-2k}(d-i)}t^{d-i}\frac{x^{n-2k}}{(n-2k)!}\\
=\ &2t\,\sum_{k\ge 1}\sum_{i\le k}E(2k,\,i-1)t^{i-1} \frac{x^{2k}}{(2k)!}
\sum_{n,\,d}\abs{\O_n(d)}t^{d}\,\frac{x^{n}}{n!}\\
=\ &2t\,O(x,t)\sum_{k\ge 1}\sum_{d\le k-1}E(2k,d)t^{d}\, \frac{x^{2k}}{(2k)!}.
\end{align*}
Combining them together, we obtain
\[
\frac{\partial O(x,t)}{\partial x} = O(x,t)\brk4{1 + 2t
\sum_{k\ge 1}\sum_{d\le k-1}E(2k,d)t^{d}\, \frac{x^{2k}}{(2k)!}}.
\]
Solving this differential equation out, we obtain \cref{dsr:O}.
\end{proof}

As a corollary, we have
\begin{cor}
For $n \geq 1$ and $0 \leq d \leq \lfloor \frac{n -1}{2}\rfloor$, we have
\[
b_n^{\des}(d)=\abs{\O_n(d)}
=\sum_{m=1}^{n}\sum_{i=0}^{m}\sum_{d_{1}+\dotsb+d_{i}=d-i,\atop
2k_{1}+\dotsb+2k_{i}=n-m}\frac{2^{i}}{m!} E(2k_{1},d_{1})\dotsm E(2k_{i},d_{i}).
\]
\end{cor}
Recall that Bidkhori and Sullivant~\cite{BS11} proved that
\[
b_{2n+1}^{\des}(n)
=\frac{E(2n+1,\,n)}{n+1},
\]
which leads to the following corollary.
\begin{cor}
For $n \geq 1$, we have
\[
\frac{E(2n+1,\,n)}{n+1}
=\sum_{m=0}^{n}\sum_{i=0}^{2m+1}
\sum_{d_{1}+\dotsb+d_{i}=n-i \atop k_{1}+\dotsb+k_{i}=n-m}
\frac{2^{i}}{(2m+1)!} E(2k_{1},d_{1})\dotsm E(2k_{i},d_{i}).
\]
\end{cor}

%
%
%
%
%


\begin{thebibliography}{99}

\bibitem{Aig08}
M. Aigner,
Enumeration via ballot numbers,
Discrete Math. 308 (2008), 2544--2563.

\bibitem{And81}
D. Andr\'e,
Sur les permutations altern\'ees,
J. Math. Pures Appl. 7 (1881), 167--184.


\bibitem{Bar87}
\'E. Barbier,
G\'en\'eralisation du probl\`eme r\'esolu par M. J. Bertrand,
Comptes Rendus de l’Acad\'emie des Sciences, Paris 105 (1887), 407.

\bibitem{BDN10}
O. Bernardi, B. Duplantier, and P. Nadeau,
A bijection between well-labelled positive paths and matchings,
S\'em. Lothar. Combin. 63 (2010), Article B63e.

\bibitem{Ber87}
J. Bertrand,
Solution d’un probl\`eme,
Comptes Rendus de l’Acad\'emie des Sciences, Paris 105 (1887), 369.

\bibitem{BS11}
H. Bidkhori and S. Sullivant, Eulerian-Catalan Numbers,
Electron. J. Combin. 18(1) (2011), 1693--1709.

\bibitem{Bona12B}
M. B\'ona,
Combinatorics of Permutations, 2nd ed.,
Discrete Math. and Its Appl.,
CRC Press, Boca Raton, FL, 2012.


\bibitem{BFW07}
F.C.S. Brown, T.M.A. Fink, and K. Willbrand,
On arithmetic and asymptotic properties of up–down numbers,
Discrete Math. 307(14) (2007), 1722--1736.

\bibitem{CFJ11}
W.Y.C. Chen, N.J.Y. Fan, and J.Y.T. Jia,
Labeled ballot paths and the Springer numbers,
SIAM J. Discrete Math. 25 (2011), 1530--1546.

\bibitem{DM47}
A. Dvoretzky and T. Motzkin,
A problem of arrangements, Duke Math. J. 14 (1947), 305--313.

\bibitem{Ent69}
R.C. Entringer,
Enumeration of permutations of $(1, \dots, n)$ by number of maxima,
Duke Math. J. 36 (1969), 575--579.

\bibitem{EFY05}
S.P. Eu, T.S. Fu, and Y.N. Yeh,
Refined Chung-Feller theorems for lattice paths,
J. Combin. Theory Ser. A 112 (2005), 143--162.

\bibitem{Fel66}
W. Feller,
An Introduction to Probability Theory and Its Applications,
Vol. 2,
Wiley, New York, 1966.

\bibitem{FS09B}
P. Flajolet and R. Sedgewick,
Analytic Combinatorics,
Camb. Univ. Press, Cambridge, 2009.

\bibitem{FS71}
D. Foata and M.P. Sch\"utzenberger,
On the principle of equivalence of Sparre Andersen,
Math. Stand. 28 (1971), 308--316.





\bibitem{Ges80}
I.M. Gessel,
A factorization for formal Laurent series and lattice path enumeration,
J. Combin. Theory Ser. A 28(3) (1980), 321--337.

\bibitem{Ges92}
\bysame,
Super ballot numbers,
J. Symbolic Comput. 14 (1992), 179--194.



\bibitem{Hum10}
K. Humphreys,
A history and a survey of lattice path enumeration,
J. Statist. Plann. Inference 140 (2010), 2237--2254.

\bibitem{Kit07}
S. Kitaev,
Introduction to partially ordered patterns,
Discrete Appl. Math. 115(8) (2007), 929--944.


\bibitem{Kyle15B}
T. Kyle Petersen,
Eulerian Numbers, Birkh\"auser, Springer, New York, 2015.


\bibitem{MSTT11}
K. Manes, A. Sapounakis, I. Tasoulas, and P. Tsikouras,
General results on the enumeration of strings in Dyck paths,
Electron. J. Combin. 18 (2011), P74, 22pp.

\bibitem{MSTT16}
\bysame,
Equivalence classes of ballot paths modulo strings of length 2 and 3,
Discrete Math. 339(10) (2016), 2557--2572.

\bibitem{NS10}
H. Niederhausen and S. Sullivan,
Pattern avoiding ballot paths and finite operator calculus,
J. Statist. Plann. Inference 140 (2010), 2312--2320.

\bibitem{Niv68}
I. Niven,
A combinatorial problem of finite sequences,
Nieuw Arch. Wiskd. (3) 16 (1968), 116--123.


\bibitem{Ren07}
M. Renault,
Four proofs of the ballot theorem,
Math. Mag. 80 (2007), 345--352.

\bibitem{Sha01}
L.W. Shapiro,
Some open questions about random walks, involutions, limiting distributions, and generating functions,
Adv. in Appl. Math. 27 (2001), 585--596.


\bibitem{She12}
V. Shevelev, Number of permutations with prescribed up-down structure as a function of two variables, Integers 12(4) (2012), 529--569.

\bibitem{OEIS}
N.J.A. Sloane,
The On-Line Encyclopedia of Integer Sequences,
{\tt http://oeis.org}, 2020.

\bibitem{Spi20}
S. Spiro,
Ballot permutations and odd order permutations,
Discrete Math. 343(6) (2020), 111869.

\bibitem{Sta97B}
R.P. Stanley,
Enumerative Combinatorics, vol. 1, in:
Cambridge Stud. Adv. Math. 49,
Camb. Univ. Press, Cambridge, 1997.


\bibitem{Tak97}
L. Tak\'{a}cs, On the ballot theorems, Advances in Combinatorial Methods and Applications to Probability and Statistics, Birkh\"{a}user, 1997.

\bibitem{WZ20}
D.G.L. Wang and J.J.R. Zhang,
A Toeplitz property of ballot permutations and odd order permutations,
Electron. J. Combin. 27(2) (2020), P2.55.

\bibitem{Whi78}
W.A. Whitworth,
Arrangements of $m$ things of one sort and $n$ things of another sort,
under certain conditions of priority,
Messenger Math. 8(1878), 105--114.

\bibitem{Wilf06B}
H. Wilf,
Generatingfunctionology,
3rd ed.,
A K Peters, Ltd., Wellesley, MA, 2006.

\bibitem{Woa01}
W.J. Woan, Uniform partitions of lattice paths and Chung-Feller generalizations, Amer. Math. Monthly 108(6) (2001), 556--559.

\bibitem{Zhuang16}
Y. Zhuang,
Counting permutations by runs,
J. Combin. Theory Ser. A 142 (2016), 147--176.

\bibitem{Zhuang17}
\bysame,
Eulerian polynomials and descent statistics,
Adv. in Appl. Math. 90 (2017), 86--144.
\end{thebibliography}
\end{document}